\documentclass[12pt,twoside]{amsart}
\usepackage{amssymb}
\usepackage{verbatim}
\usepackage{amsmath}
\usepackage{bm}
\usepackage{a4wide}
\usepackage[T1]{fontenc}
\usepackage{times}
\usepackage{amssymb,latexsym}
\usepackage{enumerate}
\usepackage{pict2e}

\makeatletter
\DeclareRobustCommand{\intprod}{%
  \mathbin{\mathpalette\int@prod{(0.1,0)(0.9,0)(0.9,0.8)}}%
}
\DeclareRobustCommand{\intprodr}{%
  \mathbin{\mathpalette\int@prod{(0.1,0.8)(0.1,0)(0.9,0)}}}

\newcommand{\int@prod}[2]{%
  \begingroup
  \sbox\z@{$\m@th#1+$}%
  \setlength\unitlength{\wd\z@}%
  \begin{picture}(1,1)
  \roundcap
  \polyline#2
  \end{picture}%
  \endgroup
}
\makeatother

\makeatletter
\newcommand{\sumprime}{\if@display\sideset{}{'}\sum%
            \else\sum'\fi}
\makeatother

\allowdisplaybreaks
\begin{document}

\numberwithin{equation}{section}

\newtheorem{theorem}{Theorem}[section]
\newtheorem{prop}[theorem]{Proposition}
\newtheorem{conjecture}[theorem]{Conjecture}
\def\theconjecture{\unskip}
\newtheorem{corollary}[theorem]{Corollary}
\newtheorem{lemma}[theorem]{Lemma}
\newtheorem{observation}[theorem]{Observation}
\newtheorem{definition}{Definition}
\newtheorem*{definition*}{Definition}
\numberwithin{definition}{section} 
\newtheorem{remark}{Remark}
\newtheorem*{note}{Note}
\def\theremark{\unskip}
\newtheorem{kl}{Key Lemma}
\def\thekl{\unskip}
\newtheorem{question}{Question}
\def\thequestion{\unskip}
\newtheorem*{example}{Example}
\newtheorem{problem}{Problem}

\thanks{Supported by National Natural Science Foundation of China, No. 11771089}

\title[Log-Hyperconvexity Index and Bergman Kernel]{Log-Hyperconvexity Index and Bergman Kernel}

 \author[Bo-Yong Chen]{Bo-Yong Chen}
 \author[Zhiyuan Zheng]{Zhiyuan Zheng}
\date{2022. 6. 16}

\address[Bo-Yong Chen]{School of Mathematical Sciences, Fudan University, Shanghai, 20043, China}

\email{boychen@fudan.edu.cn}

\address[Zhiyuan Zheng]{School of Mathematical Sciences, Fudan University, Shanghai, 20043, China}

\email{19110180018@fudan.edu.cn}

\begin{abstract}
We obtain a quantitative estimate of Bergman distance when $\Omega \subset \mathbb{C}^n$ is a bounded domain with log-hyperconvexity index $\alpha_l(\Omega)>\frac{n-1+\sqrt{(n-1)(n+3)}}{2}$, as well as the $A^2(\log A)^q$-integrability of the Bergman kernel $K_{\Omega}(\cdot, w)$ when $\alpha_l(\Omega)>0$.



\end{abstract}

\maketitle

\section{Introduction}

We say a domain $\Omega \subset \mathbb{C}^n$ is hyperconvex if there exists a continuous negative plurisubharmonic (psh) exhaustion function $\rho$ on $\Omega$, i.e., $\{ \rho < c  \} \Subset \Omega$ for every $c<0$. 
It is one of the core concepts in the function theory of several complex variables, which can be traced back to 1974, when Stehl\'e \cite{Stehle1974} first proposed it in order to study the Serre problem. Clearly, hyperconvexity implies pseudoconvexity, but the converse fails. Thus it is interesting to ask when a pseudoconvex domain is hyperconvex.
A large literature of positive results exists (see, e.g., \cite{ChenHolder,Demailly1987,DiederichFornaess1975,DiederichFornaess1977,Harrington2007,AHP2014,KerzmanRosay1981,OhsawaSibony1998}).   
 Among them, the weakest regularity assumption on the boundary is that the boundary can be written locally as the graph of a H\"older continuous function  (cf. \cite{ChenHolder}).

The quantitative characterization of hyperconvexity starts from the seminal work of Diederich-Fornaess \cite{DiederichFornaess1977}, that for every bounded pseudoconvex domain $\Omega$ with $C^2$- boundary one has
$$
\eta(\Omega) := \sup \{ \eta \geq 0; \exists \hspace{0.1cm} \rho \in C(\Omega)\cap PSH^{-}(\Omega) \text{ s.t.} -\rho \asymp \delta^{\eta} \}>0, 
$$ 
where $\delta$ denotes the boundary distance and $PSH^{-}(\Omega)$ denotes the set of negative psh functions on $\Omega$. The quantity $\eta(\Omega)$ is also called the Diederich-Fornaess index (D-F index) of $\Omega$, and it has been studied by a number of authors (see, e.g.,  \cite{BerndtssonCharpentier2000,DiederichFornaess1977,FornaessHerbig2008,FuShaw2014,Harrington2007,Liu2019,OhsawaSibony1998,Sibony1987}). In 2017, the first author \cite{Chen2017} introduced a related concept, the so-called hyperconvexity index of a bounded domain $\Omega$, defined as
$$
\alpha(\Omega):= \sup \{ \alpha \geq 0 ; \exists \hspace{0.1cm} \rho \in C(\Omega)\cap PSH^{-}(\Omega) \text{ s.t.} -\rho \lesssim \delta^{\alpha} \}. 
$$
Obviously, $\alpha(\Omega) \geq \eta(\Omega)$. These concepts have numerous applications in the study of the Bergman kernel and metric.

In this spirit, we introduce the following
\begin{definition}
For   a bounded domain $\Omega \subset \mathbb{C}^n$, we define
$$
\alpha_{l}(\Omega):=  \sup \{ \alpha \geq 0; \exists \hspace{0.1cm} \rho \in C(\Omega)\cap PSH^{-}(\Omega) \text{ s.t. }  -\rho \lesssim (-\log \delta)^{-\alpha}      \text{ near } \partial \Omega      \}
$$
as the log-hyperconvexity index of $\Omega$.
\end{definition}
This concept is motivated by the recent result of the first author \cite{Chen2017} that every bounded pseudoconvex H\"older domain \emph{locally has a positive log-hyperconvexity index}. It remains open whether every bounded pseudoconvex H\"older domain has a (global) positive log-hyperconvexity index (see \cite{Xiong2021} for some partial results). Note that there are various examples with $\alpha(\Omega)=0$ while $\alpha_l(\Omega)>0$ (see Appendix).

It is well-known that the growth of negative psh exhaustion functions relates closely to the estimate of the Bergman distance $d_B(z_0,z)$ from a fixed point $z_0$ to $z$.
 If $\Omega$ is hyperconvex, then $\lim_{z\to \partial \Omega} d_B(z_0, z)=\infty$ (cf.  \cite{BlockiPflug1998} or \cite{Herbort1999}). Diederich-Ohsawa \cite{DiederichOhsawa1995} obtained the quantitative estimate $$d_B(z_0,z)\gtrsim \log |\log \delta(z)|$$ for all $z$ sufficiently close to $\partial \Omega$ when there exists $\rho \in C(\Omega)\cap PSH^{-}(\Omega)$ s.t. $ \delta^\beta\lesssim -\rho\lesssim \delta^\alpha$, where $\beta\ge \alpha >0$.  In case $\alpha=\beta$, i.e., $\eta(\Omega)>0$, Blocki \cite{Blocki2004} improved the previous estimate to $$d_B(z_0,z)\gtrsim |\log\delta(z)|/\log |\log \delta(z)|.$$  In \cite{Chen2017}, the first author showed that  Blocki's estimate remains valid under the weaker condition $\alpha(\Omega)>0$.  Here we shall prove the following

\begin{theorem}\label{theorem1.3}
Let $\Omega$ be a bounded pseudoconvex domain in $\mathbb C^n$. If for every $p\in \partial \Omega$ there is an open neighborhood $D$ of $p$ such that 
$$
\alpha_l(\Omega \cap D)> \frac{n-1+\sqrt{(n-1)(n+3)}}{2},
$$ 
then
\begin{equation}\label{eq:1.1}
d_B(z_0,z) \gtrsim \log\log |\log \delta(z)|
\end{equation}
for all $z$ sufficiently close to $\partial \Omega$.
\end{theorem}

\begin{problem}
Does the condition $\alpha_l(\Omega)>0$ imply $(\ref{eq:1.1})$?
\end{problem}
\begin{remark}Theorem \ref{theorem1.3} shows that the answer to this problem is affirmative when $n=1$. 
\end{remark}
Theorem \ref{theorem1.3} combined with the main result (and its proof) in \cite{ChenHolder} gives

\begin{corollary}\label{corollary1.1}
Let $\Omega$ be a bounded pseudoconvex domain in $\mathbb C^n$ such that $\partial \Omega$ can be written locally as the graph of a H\"older continuous function of order $\gamma\ge \frac{\sqrt{(n-1)(n+3)}-n+1}{2}$. Then $(\ref{eq:1.1})$ holds.
\end{corollary}

Note that positivity of $\alpha(\Omega)$ can be used to obtain the $L^p$-integrability of $K_{\Omega}(\cdot, w)$ for fixed $w$ when $2<p<2+\frac{2\alpha(\Omega)}{2n-\alpha(\Omega)}$ (cf. \cite{Chen2017}). Analogously, we have

\begin{theorem} \label{theorem1.2}
Let $\Omega$ be a pseudoconvex domain in  $\mathbb{C}^n$  with $\alpha_{l}(\Omega)>0$. For fixed $w\in \Omega$, we have 
$$K_{\Omega}(\cdot,w) \in A^2(\log A)^q(\Omega):=\left\{ f\in \mathcal{O}(\Omega); \int_{\Omega}|f|^2(\log^{+} |f|)^q \mathrm{d}\lambda < \infty  \right\}$$
 for all $0<q<\alpha_{l}(\Omega)$, where $\log^{+}t=\max\{ \log t,0 \}$ and $\mathrm{d}\lambda$ denotes the Lebesgue measure.
\end{theorem}
Note that $A^p(\log A)^q(\Omega)$ is a natural analogue of the $L\log L$ space, which  plays an important role in harmonic analysis and has been studied by many authors (cf. \cite{Burkholder1962,CalderonZygmund1952,CalderonZygmund1954,Gundy1969,Stein1969} etc).

As an application, we obtain the following
\begin{corollary} \label{corollary6.4}
If  $\Omega$ is a bounded pseudoconvex domain in $\mathbb{C}^n$ with $\alpha_{l}(\Omega) >0$, then  $A^2(\log A)^q(\Omega)$ lies dense in $A^2(\Omega)$ for every $0<q<\alpha_{l}(\Omega)$. 
\end{corollary}

Barrett \cite{Barrett1984} found for each $k \in \mathbb{Z}^{+}$ a smooth bounded (non-pseudoconvex) domain $\Omega_k \subset \mathbb{C}^2$, such that the Bergman space $A^p(\Omega_k)$ is not dense in $A^2(\Omega_k)$ when $p\ge 2+1/k$. A minic of his construction gives

\begin{theorem}
For every $s \in (0,1)$ and $q \geq \frac{1}{s}-1$,  there exists a smooth bounded (non-pseudoconvex) domain $\Omega^{(s)} \subset \mathbb{C}^2$ such that the following properties hold:

$(1)$ $A^2(\log A)^q(\Omega^{(s)})$ is not dense in $A^2(\Omega^{(s)})$;

$(2)$  $K_{\Omega^{(s)}}(\cdot, w) \notin A^2(\log A)^q(\Omega^{(s)})$ for some $w\in \Omega^{(s)}$.
\end{theorem}
\begin{remark}Note that  if $\Omega \subset \mathbb{C}^n$ is bounded, $p>2$ and $q>0$, then
$$
A^p(\Omega) \subset A^2(\log A)^q(\Omega) \subset A^2(\Omega).
$$
Thus Theorem 1.5 implies that there exists a smooth bounded domain $\Omega \subset \mathbb{C}^2$ such that 
$\bigcup_{p>2}A^p(\Omega)$ is not dense in $A^2(\Omega)$. 

\end{remark}

\section{The proof of Theorem 1.1}

The classical method of etimating the Bergman metric is to use the pluricomplex Green function (see, e.g., \cite{Blocki2004,Chen2017,Herbort2008}). 
By the well-known localization principle for the Bergman metric, we only need to deal with the case 
$$\alpha_l(\Omega) > \frac{n-1+\sqrt{(n-1)(n+3)}}{2}.
$$
Let $\varrho=\varrho_{\overline{B}}:=\sup\{ u\in PSH^{-}(\Omega); u|_{\overline{B}}<-1 \}$ be the relative extremal function of a fixed closed ball $\overline{B}\subset \Omega$. By the extremal property of $\varrho$ there is a constant $C_{\alpha}>0$ for each  $\alpha \in (0,\alpha_{l}(\Omega))$,  such that 
$$
-\varrho(z)\le C_{\alpha}(-\log\delta(z))^{-\alpha}
$$
for all $z$ sufficiently close to $\partial \Omega$. 


\begin{lemma}\label{lemma3.2}
For every $r>1$  there exists a constant $\varepsilon_{r}\ll 1$ such that 
$$
\varrho(z_2)\ge  r\varrho(z_1) -C_{\alpha}(-\log |z_1-z_2|)^{-\alpha}
$$
for all $z_1,z_2\in \Omega$ with $|z_1-z_2|\le \varepsilon_r$. 
\end{lemma}
\begin{proof}
Set $\varepsilon:=|z_1-z_2|$, $\Omega'=\Omega-(z_1-z_2)$ and
$$
u(z) = \left\{
\begin{array}{ll}
\varrho(z) & z\in \Omega\backslash \Omega'\\
\max\{\varrho(z),r\varrho(z+z_1-z_2) - C_{\alpha} (-\log \varepsilon)^{-\alpha}\}
  & z\in \Omega\cap \Omega'.
  \end{array}
  \right.
  $$
  For every $z\in \Omega\cap \partial \Omega'$ we have $\delta(z)\le \varepsilon$, so that
  $$
  \varrho(z) \ge -C_{\alpha} (-\log \varepsilon)^{-\alpha} \ge r\varrho(z+z_1-z_2) - C_{\alpha} (-\log \varepsilon)^{-\alpha}.
    $$
    Thus $u\in PSH^-(\Omega)$. On the other hand, for $\varepsilon\le \varepsilon_r\ll 1$ we have
    $$
    \varrho(z+z_1-z_2)\le -1/r, \forall\,z\in \overline{B},
    $$
    in view of the continuity of $\varrho$. Thus $u|_{\overline{B}}\le -1$. Since $z_2=z_1-(z_1-z_2)\in \Omega\cap \Omega'$, we infer from the extremal property of $\varrho$ that
    $$
    \varrho(z_2) \ge u(z_2)\ge r\varrho(z_1)-C_{\alpha}(-\log \varepsilon)^{-\alpha}.
    $$
\end{proof}

Let $g_{\Omega}(z,w)$ be the pluricomplex Green function of $\Omega$. For $c>0$ we define
$$
A_{\Omega}(w,-c):=\{ z\in \Omega ; g_{\Omega}(z,w)\leq -c \}
$$

\begin{lemma}\label{proposition3.3}
There exists a constant $C\gg 1$ such that for any $w\in \Omega$
$$
A_\Omega(w,-1) \subset \left\{z\in \Omega; \varrho(z)<-C^{-1}\mu(w)\right\},
$$
where $\mu(w):=(-\varrho(w))^{1+\frac{1}{\alpha}}$.
\end{lemma}

\begin{proof}
Apply Lemma \ref{lemma3.2} with $r=\frac{3}{2}$, we conclude that if $\varrho(z)=\varrho(w)/2$ and $|z-w|<\varepsilon_{1/2},$ then for suitable $C_1\gg 1$
$$
C_1(-\log |z-w|)^{-\alpha} \ge \frac{3}{2}\varrho(z)-\varrho(w)=-\varrho(w)/4.
$$ 
When $|z-w|\geq \varepsilon_{1/2},$ it is also easy to see that
$$
C_1(-\log |z-w|)^{-\alpha} \geq -\varrho(w)/4
$$
if $C_1$ is large enough. Thus for all $z$ with $\varrho(z)=\varrho(w)/2$ we have
$$
\log \frac{|z-w|}R\ge -C_2 (-\varrho(w))^{-1/\alpha}
$$
for suitable $C_2>0$, where $R$ denotes the diameter of $\Omega$. It follows that 
$$
\psi(z) = \left\{
\begin{array}{ll}
\log(|z-w|/R) & \text{if\ } \varrho(z)    \le\varrho(w)/2\\
\max\{\log(|z-w|/R), 2C_2 (-\varrho(w))^{-1-1/\alpha} \varrho(z)\}
  & \text{otherwise}
  \end{array}
  \right.
  $$
  is a negative psh function on $\Omega$ which has a logarithmic pole at $w$. Thus for $\varrho(z)\ge \varrho(w)/2 $ we have
  $$
  g_\Omega(z,w) \ge \psi(z) \ge 2C_2 (-\varrho(w))^{-1-1/\alpha} \varrho(z),
    $$
    so that
    $$
    A_\Omega(w,-1)\cap \{\varrho\ge \varrho(w)/2 \} \subset \left\{\varrho< -C^{-1} \mu(w)\right\}
    $$
    whenever $C\gg 1$. Since 
    $$
    \{\varrho< \varrho(w)/2 \} \subset \left\{\varrho< -C^{-1} \mu(w)\right\}
    $$
     for $C\gg 1$, we conclude the proof.
  \end{proof}

\begin{prop}\label{proposition3.1}
Let $\alpha >\frac{n-1+\sqrt{(n-1)(n+3)}}{2}$ and $n-1<\beta < \frac{\alpha^2}{\alpha+1}$. Then there exists a constant $C\gg 1$ such that 
 $$
  A_\Omega(w,-1) \subset \left\{z\in \Omega; -C \nu(w) < \varrho(z)<-C^{-1}\mu(w) \right\},
 $$
   where $\nu(w):=(-\varrho(w))^{\left(\frac{1}{n}-\frac{1}{\beta}(1-\frac{1}{n})\right)\frac{n\alpha}{1+\alpha}}$.
  \end{prop} 

\begin{remark}
Note that $\alpha >\frac{n-1+\sqrt{(n-1)(n+3)}}{2}$ implies $\frac{\alpha^2}{\alpha+1}>n-1$, and $\beta >n-1$ implies that the exponent in $\nu$ is positive.
\end{remark}
  The proof of Proposition \ref{proposition3.1} is based on the following
  \begin{lemma}[cf. \cite{Blocki2004}]\label{lm:4.2}
 Let $\Omega\subset {\mathbb C}^n$ be a bounded hyperconvex domain.  Suppose $\zeta,w$ are two points in $\Omega$ such that the closed balls $\overline{B}(\zeta,\varepsilon),\,\overline{B}(w,\varepsilon)\subset \Omega$ and $\overline{B}(\zeta,\varepsilon)\cap \overline{B}(w,\varepsilon)=\emptyset$. Then there exists $\tilde{\zeta}\in \overline{B}(\zeta,\varepsilon)$ such that
\begin{equation}\label{eq:2.1}
|g_\Omega(\tilde{\zeta},w)|^n \le n!(\log R/\varepsilon)^{n-1}|g_\Omega(w,\zeta)|.
\end{equation}
\end{lemma}

\begin{proof}[Proof of Proposition \ref{proposition3.1}]
For fixed $w\in \Omega$ which is sufficiently close to $\partial \Omega$, we set
$$
\varepsilon:=\exp({-(-\varrho(w))^{-1/\beta}}),
$$
 i.e., 
$
-\varrho(w)=(-\log \varepsilon)^{-\beta}$. Since $-\varrho(w)\le C_0(-\log \delta(w))^{-\alpha}$ and $\beta<\alpha$, it follows that if $\delta(w)$ is sufficiently small then
$$
2^\alpha C_0(-\log \varepsilon)^{-\alpha}\le (-\log \varepsilon)^{-\beta}\le C_0(-\log \delta(w))^{-\alpha},
$$
so that $\delta(w)\ge \sqrt{\varepsilon}$ and we have $B(w,\varepsilon)\subset\subset \Omega$. If $\delta(z)\le \varepsilon$, then
\begin{equation}\label{eq:2.2}
-\varrho(z) \le C_0 (-\log \delta(z))^{-\alpha}\le C_0 (-\log \varepsilon)^{-\alpha} = C_0 (-\varrho(w))^{\alpha /\beta}\le -\varrho(w)/2.
\end{equation}
By the proof of  Proposition \ref{proposition3.3}, we have
\begin{equation}\label{eq:2.3}
g_\Omega(z,w) \ge C_1 (-\varrho(w))^{-1-1/\alpha} \varrho(z)
\end{equation}
for some constant $C_1\gg 1$.
This combined with (\ref{eq:2.2}) gives
\begin{equation}\label{eq:2.4}
\sup_{\delta\le \varepsilon} |g_\Omega(\cdot,w)| \le C_0C_1 (-\varrho(w))^{\alpha / \beta -1-1/\alpha}.
\end{equation}
By Lemma \ref{lemma3.2}, we see that if $\varrho(\zeta)\le 2\varrho(w)$ then for suitable $C_2>0$,
$$
C_2(-\log |\zeta-w|)^{-\alpha} \ge  \frac{3}{2}\varrho(w)-\varrho(\zeta)\geq -\varrho(w)/2=(-\log \varepsilon)^{-\beta} /2 > 2^\alpha C_2 (-\log \varepsilon)^{-\alpha}
$$
whenever $\delta(w)\ll 1$, so that $|\zeta-w|>\sqrt{\varepsilon}$ and 
$$
\overline{B}(\zeta,\varepsilon)\cap \overline{B}(w,\varepsilon)=\emptyset.
$$
It follows from Lemma \ref{lm:4.2} that there exists $\tilde{\zeta}\in \overline{B}(\zeta,\varepsilon)$ such that (\ref{eq:2.1}) holds. We also need the following known inequalities (cf. \cite{BlockiPotential}, Proposition 3.3.3)
\begin{equation}\label{eq:2.5}
g_\Omega(z,w) \ge (\log R/\varepsilon)\cdot \varrho_{\overline{B}(w,\varepsilon)},\ \ \ \forall\,z\in \Omega\backslash {B}(w,\varepsilon)
\end{equation}
\begin{equation}\label{eq:2.6}
g_\Omega(z,w) \le (\log \delta(w)/\varepsilon)\cdot \varrho_{\overline{B}(w,\varepsilon)},\ \ \ \forall\,z\in \Omega.
\end{equation}
By (\ref{eq:2.4}) and (\ref{eq:2.6}), we have
\begin{equation}\label{eq:2.7}
\sup_{\delta\le \varepsilon} |\varrho_{\overline{B}(w,\varepsilon)}| \le C_0C_1(\log \delta(w)/\varepsilon)^{-1} (-\varrho(w))^{\alpha /\beta-1-1/\alpha}=:\tau_\varepsilon(w). 
\end{equation}
Set $\widetilde{\Omega}:=\Omega-(\tilde{\zeta}-\zeta)$ and
$$
v(z) = \left\{
\begin{array}{ll}
\varrho_{\overline{B}(w,\varepsilon)}(z) & z\in \Omega\backslash \widetilde{\Omega}\\
\max\{ \varrho_{\overline{B}(w,\varepsilon)}(z), \varrho_{\overline{B}(w,\varepsilon)}(z+\tilde{\zeta}-\zeta)- \tau_\varepsilon(w)\}
  & z\in \Omega\cap \widetilde{\Omega}.
  \end{array}
  \right.
  $$
  By (\ref{eq:2.7}), we see that $v$ is a well-defined negative psh function on $\Omega$. On the other hand, since  
  $$
  \varrho_{\overline{B}(w,\varepsilon)}(z)\le \frac{\log \left(|z-w|/\delta(w)\right)}{\log R/\varepsilon},\ \ \ z\in \Omega\backslash B(w,\varepsilon),
  $$
  in view of (\ref{eq:2.5}), and $z+\tilde{\zeta}-\zeta\in \overline{B}(w,2\varepsilon)$ if $z\in \overline{B}(w,\varepsilon)$,   we have
  $$
  v|_{\overline{B}(w,\varepsilon)}\le -\frac{\log (\delta(w)/(2\varepsilon))}{\log R/\varepsilon}
  $$
  so that 
  $$
  \varrho_{\overline{B}(w,\varepsilon)}(\tilde{\zeta}) - \tau_\varepsilon(w) \le v(\zeta) \le \frac{\log (\delta(w)/(2\varepsilon))}{\log R/\varepsilon}\cdot \varrho_{\overline{B}(w,\varepsilon)}(\zeta).
  $$
  This combined with (\ref{eq:2.5}) and (\ref{eq:2.6}) gives
  \begin{eqnarray}
\notag  g_\Omega(\zeta,w) & \ge & \frac{(\log R/\varepsilon)^2}{\left(\log \delta(w)/\varepsilon \right) \cdot \left( \log \delta(w)/(2\varepsilon)\right)} \left( g_\Omega(\tilde{\zeta},w)-C_0C_1(-\varrho(w))^{\alpha /\beta-1-1/\alpha}\right)\\
\notag  & \ge & C_3 \left( g_\Omega(\tilde{\zeta},w)-C_0C_1(-\varrho(w))^{\alpha / \beta-1-1/\alpha}\right)\\
\notag  & \ge & -C_3 (n!)^{1/n}(\log R/\varepsilon)^{1-1/n} |g_\Omega(w,\zeta)|^{1/n}-C_0C_1C_3(-\varrho(w))^{\alpha / \beta-1-1/\alpha} \\
  &\ge & -C_4 (-\varrho(w))^{\frac{1}{n}-(1-\frac{1}{n})\frac{1}{\beta}}(-\varrho(\zeta))^{-\frac1n(1+\frac{1}{\alpha})}-C_0C_1C_3(-\varrho(w))^{\alpha /\beta-1-1/\alpha}     \end{eqnarray}
  for some constants $C_4\gg C_3>0$, in view of  (\ref{eq:2.3}).  Since  $\alpha / \beta -1-1/\alpha >0$, we see that if $\delta(w)\ll1$ then  $C_0C_1C_2(-\varrho(w))^{\alpha / \beta -1-1/\alpha } \le 1/2$, so that   
  $$
  A_\Omega(w,-1)\cap \{\varrho\le 2\varrho(w)\}\subset \{\varrho>-C\nu(w)\}
  $$      
  for some $C\gg 1$. On the other hand, we have $\{\varrho>2\varrho(w)\}\subset \{\varrho>-C\nu(w)\}$ for $C\gg 1$ since the exponent in $\nu$ is less than one. Thus the proof is complete.  
 \end{proof}



\begin{proof}[Proof of Theorem \ref{theorem1.3}]

 Let $z\in \Omega$ be  sufficiently close to $\partial \Omega$. We may choose  a Bergman geodesic  jointing $z_0$ to $z$,  and a finite number of points $\{z_k\}_{k=1}^m$ on this geodesic with the following order
  $$
  z_0 \rightarrow z_1\rightarrow z_2\rightarrow \cdots \rightarrow z_m \rightarrow z,
  $$
  where
  $$
  C^{-1} \mu(z_k) =C\nu(z_{k+1}) \ \ \ \text{and}\ \ \  C^{-1}\mu(z_m)\le -\varrho(z)\le C\nu(z_m)
   $$
   for some $C\gg 1$ so that Proposition \ref{proposition3.1}  hold. 
  Thus we have 
  $$
  \{g_\Omega(\cdot,z_k)\le -1\}\cap \{g_\Omega(\cdot,z_{k+1})\le -1\}=\emptyset
  $$
   so that  $d_B(z_k,z_{k+1})\ge c_1>0$ for all $k$ in view of Theorem 1.1 in \cite{Blocki2004}.
  
   Set $\gamma:=\left(\frac{1}{n}-\frac{1}{\beta}(1-\frac{1}{n})\right)\frac{n\alpha}{1+\alpha}$.
    Note that
   \begin{eqnarray*}
   \log (-\varrho(z_0)) & = & \frac{\gamma\alpha}{1+\alpha} \log (-\varrho(z_1)) + \frac{2\alpha}{\alpha+1} \log C =\cdots \\
   & = &  \left( \frac{\gamma\alpha}{1+\alpha} \right)^m  \log (-\varrho(z_m)) + \frac{1-\left(\frac{\gamma\alpha}{1+\alpha}\right)^m}{1-\frac{\gamma\alpha}{1+\alpha}} \frac{2\alpha}{\alpha+1} \log C       \end{eqnarray*}
      Thus we have
      $$
      m\asymp \log |\log (-\varrho (z_m))| \asymp \log |\log (-\varrho (z))| \gtrsim \log\log |\log \delta(z)|,
      $$
      so that 
    $$
  d_B(z_0,z)  \ge  \sum_{k=1}^{m-1} d_B(z_k,z_{k+1})\ge c_1(m-1) \gtrsim \log\log |\log \delta(z)|.
    $$
 \end{proof}

\section{Bergman Kernel and $A^p(\log A)^q$}

We first introduce some basic facts about the $A^p(\log A)^q$ space. For $p,q>0$, let
$L^p(\log L)^q(\Omega)$ be the set of measurable complex-valued  functions $f$ on $ \Omega $ such that 
$$ \int_{\Omega} |f|^p(\log^{+} |f|)^q  \mathrm{d}\lambda< \infty       .       $$ 
Set
$$\|f\|_{L^p(\log L)^q(\Omega)}=\inf \left\{ s>0;\int_{\Omega} \left(\frac{|f|}{s}\right)^p\left( \log^{+}\frac{|f|}{s}\right)^q  \mathrm{d}\lambda\leq 1 \right\} $$
and
$$
A^p(\log A)^q(\Omega):=L^p(\log L)^q (\Omega) \cap \mathcal{O}(\Omega).
$$ 
When $q=0$, $\| \cdot \|_{L^p(\log L)^q(\Omega)}$ is the usual $L^p$ norm.

\begin{prop}\label{prop:3.1}
If $\Omega \subset \mathbb{C}^n$ is a bounded domain, then the following properties hold:

$(1)$ $L^p(\log L)^q(\Omega)$ is a linear space;

$(2)$ $L^p(\log L)^q(\Omega) \subset L^p(\Omega)$ and $\|\cdot\|_{L^p(\Omega)} \lesssim \|\cdot\|_{L^p(\log L)^q(\Omega)}$;

$(3)$ If $p,q \geq 1$, then $L^p(\log L)^q(\Omega)$ is a Banach space with norm $\|\cdot \|_{L^p(\log L)^q(\Omega)}$;

$(4)$ If $p,q \geq 1$, then $A^p(\log A)^q$ is a closed subspace of $L^p(\log L)^q(\Omega)$.
\end{prop}

The arguments are standard, and we include the proof for the sake of completeness. 

\begin{proof}

$(1)$ Given $f,g \in L^p(\log L)^q(\Omega)$ and $c \in \mathbb{C}-\{0\}$, we have
\begin{eqnarray*}
\int_{\Omega} |cf|^p(\log^{+}|cf|)^q \mathrm{d}\lambda &=& \left( \int_{\{|f|\leq |c| \}}+\int_{\{ |f|>  |c|\}}\right)|cf|^p(\log^{+}|cf|)^q \mathrm{d}\lambda  \\
& \leq& \mathrm{const.} + \int_{\{ |f|> |c|\}}|c|^p|f|^p(\log^{+} |c|+\log^{+} |f|)^q \mathrm{d}\lambda  \\
&\leq& \mathrm{const.}+  \int_{\Omega} |c|^p|f|^p(2\log^{+} |f|)^q \mathrm{d}\lambda \\
&= & \mathrm{const.}+ 2^q|c|^p\int_{\Omega}|f|^p(\log^{+} |f|)^q \mathrm{d}\lambda \\
&<& +\infty,
\end{eqnarray*}
and 
\begin{eqnarray*}
&&\int_{\Omega}\left|\frac{f+g}{2}\right|^p\left(\log^{+} \left| \frac{f+g}{2} \right|  \right)^q \mathrm{d}\lambda \\
&&\hspace{1.2cm}\leq  \int_{\Omega}\left(\frac{|f|+|g|}{2}\right)^p\left(\log^{+} \left( \frac{|f|+|g|}{2} \right)  \right)^q  \mathrm{d}\lambda\\
&&\hspace{1.2cm}\leq \left( \int_{\{|f|\geq |g|\}} +\int_{\{|f|< |g|\}} \right)\left(\frac{|f|+|g|}{2}\right)^p\left(\log^{+} \left( \frac{|f|+|g|}{2} \right)  \right)^q  \mathrm{d}\lambda\\
&&\hspace{1.2cm}\leq \int_{\Omega}|f|^p(\log^{+} |f|)^q\mathrm{d}\lambda +\int_{\Omega} |g|^p(\log^{+} |g|)^q \mathrm{d}\lambda \\
&&\hspace{1.2cm}< \infty.
\end{eqnarray*}
Hence $A^p(\log A)^q(\Omega)$ is closed under scalar multiplication and addition, which implies (1).

$(2)$
If $\|f\|_{L^p(\log L)^q(\Omega)}=0$, then for every $\varepsilon >0$ we have
$$
\int_{\Omega}|f|^p\left(\log^+ \frac{|f|}{\varepsilon}\right)^q\mathrm{d}\lambda \leq \varepsilon^p,
$$
since $h(t)=\int_{\Omega}  \left( |f|/t \right)^p\left( \log^{+} |f|/t \right)^q$ is nonincreasing for $t>0$. In particular, we have $\|f\|_{L^p(\Omega \cap \{ |f|>\varepsilon e \}) } \leq \varepsilon$. Letting $\varepsilon \to 0$, we get $\|f\|_{L^p(\Omega)}=0$. Now suppose  
$$\|f\|_{L^p(\log L)^q(\Omega)}=\tau>0.$$ It follows again from the monotonicity of $h(t)$ that
$$
\int_{\Omega} \left( \frac{|f|}{\tau+\varepsilon}  \right)^p \left( \log^{+} \frac{|f|}{\tau+\varepsilon}  \right)^q\mathrm{d}\lambda \leq 1, \, \, \, \forall \, \varepsilon >0.
$$
Hence
\begin{eqnarray*}
\int_{\Omega} |f|^p\mathrm{d}\lambda &\leq& \int_{ \{ |f| \leq (\tau+\varepsilon) e\}}|f|^p\mathrm{d}\lambda + \int_{ \{ |f|>(\tau+\varepsilon) e\}} |f|^p \left( \log^{+} \frac{|f|}{\tau+\varepsilon}\right)^q\mathrm{d}\lambda \\
&\leq& (|\Omega|e^p +1)(\tau+\varepsilon)^p ,\, \, \, \forall \, \varepsilon >0,
\end{eqnarray*}
from which $(2)$ immediaely follows.

$(3)$
The assertion follows directly from (1) and Theorem 10 in Chapter 3 of \cite{Rao} with $g(t)=|t|^p(\log^{+}|t|)^q$ as a Young function. 

$(4)$
Let $\{ h_n\}$ be a Cauchy sequence in $A^p(\log A)^q(\Omega)$. By $(2)$ and $(3)$  there is an $h \in L^p(\log L)^q(\Omega)$ such that $h_n \to h$ under $\|\cdot\|_{L^p(\Omega)}$ and $\|\cdot \|_{L^p(\log L)^q(\Omega)}$. Since $A^p(\Omega)$ is complete, we  have $$h \in A^p(\Omega) \cap L^p(\log L)^q(\Omega) =A^p(\log A)^q(\Omega).$$ Thus $A^p(\log A)^q(\Omega)$ is a closed subspace of $L^p(\log L)^q(\Omega)$.
\end{proof}

\begin{remark}Note that  for $p\ge 1, q\ge 0$, $g(t)=t^p(\log^{+} t)^q$ is convex on $[x_0,+\infty)$ for some $x_0>0$. Thus $$h(t):=\max\{0, g(t)-g(t_0)\}$$  is  convex on $\mathbb{R}^+$ and satisfies
$$\int_{\Omega}g(|f|)\mathrm{d}\lambda < \infty \Leftrightarrow \int_{\Omega}h(|f|)\mathrm{d}\lambda < \infty,
$$ 
since $\Omega$ is bounded. So if we choose $h$ as a Young function, then $A^p(\log A)^q(\Omega)$ is still a Banach space.
\end{remark}




In order to prove Theorem 1.3, we need the following result from \cite{Chen2017}.

\begin{prop}\label{pr:2.1}
Let $\Omega \subset \mathbb{C}^n$ be a pseudoconvex domain. Let $\rho$ be a continuous negative plurisubharmonic function on $\Omega$. Set$$\Omega_t =\{ z\in \Omega ; -\rho(z) >t      \},$$where $t>0$. Let $a>0$ be given. For every $r\in (0,1)$, there exist constants $\varepsilon_r, C_r>0$, such that 
$$
\int_{\{-\rho \leq \varepsilon\}} |K_{\Omega}(\cdot,w)|^2\mathrm{d}\lambda \leq C_r K_{\Omega_a}(w) \left( \frac{\varepsilon}{a} \right)^r
$$
for all $w \in \Omega_a$ and $\varepsilon < \varepsilon_r a$.
\end{prop}

\begin{proof}[Proof of Theorem \ref{theorem1.2}]
For every $\alpha \in (0,\alpha_{l}(\Omega))$, there exists a constant $C_\alpha >0$ such that
\begin{equation}\label{eq:2.5.1}
-\varrho \leq C_{\alpha}(-\log \delta)^{-\alpha},  
\end{equation}
where $\varrho=\varrho_{\overline{B}}$ is the relative extremal function of a fixed closed ball $\overline{B}\subset \Omega$.
From Proposition \ref{pr:2.1}, we have
$$
\int_{\{-\varrho \leq \varepsilon\}}|K_{\Omega}(\cdot,w)|^2\mathrm{d}\lambda \leq C\varepsilon^r
$$
for all $0<r<1$, where $w \in \Omega$ is fixed. Here and what in follows we use $C$ to denote all constants depending only on $\alpha , r,w$ and $\Omega$.  
By (\ref{eq:2.5.1}), we have 
$$\{ (-\log \delta)^{-\alpha} \leq \varepsilon  \} \subset \{ -\varrho \leq C_{\alpha}\varepsilon  \},$$ so that
\begin{equation}\label{eq:2.5.2}
\int_{\{(-\log \delta)^{-\alpha} \leq \varepsilon\}} |K_{\Omega}(\cdot,w)|^2\mathrm{d}\lambda \leq C\varepsilon^r.
\end{equation}
Since $B(z,\delta(z)) \subset \{ \delta \leq 2 \delta(z)  \} =\{ (-\log \delta)^{-\alpha}\leq (-\log (2\delta(z)))^{-\alpha}  \}$, we infer from the mean value inequality that
 \begin{eqnarray}
\notag |K_{\Omega}(z,w)|^2 &\leq &C\delta(z)^{-2n}\int_{\{(-\log \delta)^{-\alpha} \leq (-\log (2\delta(z)))^{-\alpha}\}}|K_{\Omega}(\cdot,w)|^2\mathrm{d}\lambda \\
&\leq &C \delta(z)^{-2n} (-\log (2\delta(z)))^{-\alpha r},   \label{pointwise}
\end{eqnarray}
which implies
$$
\log^{+}|K_{\Omega}(\cdot,w)|\leq \max \left\{ 0,C-n\log \delta(\cdot) -\frac{\alpha r}{2} \log(-\log 2\delta(\cdot)) \right\}.
$$
When $2^{-k-1}\le (-\log \delta(z))^{-\alpha}< 2^{-k} \,( \text{here} \,k \geq k_0$ for some $k_0)$, we have
\begin{equation}\label{eq:2.5.3}
\log^{+}|K_{\Omega}(z,w)| \leq C\cdot 2^{k/\alpha}.
\end{equation}
Hence 
\begin{eqnarray*}
&&\int_{\Omega}|K_{\Omega}(\cdot,w)|^2(\log^{+}|K_{\Omega}(\cdot,w)|)^{q} \mathrm{d}\lambda\\
&&\hspace{1cm} \leq  \left(\int_{\{(-\log \delta)^{-\alpha}> 2^{-k_0}\}} + \sum_{k=k_0}^{\infty}\int_{\{2^{-k-1}\le(-\log \delta)^{-\alpha}<2^{-k}\}}\right)|K_{\Omega}(\cdot,w)|^2(\log^{+}|K_{\Omega}(\cdot,w)|)^{q} \mathrm{d}\lambda \\
&&\hspace{1cm}\leq_{(\ref{eq:2.5.3})} C +\sum_{k=k_0}^{\infty} C\cdot 2^{\frac{kq}{\alpha}}\int_{\{(-\log \delta)^{-\alpha}<2^{-k}\}}|K_{\Omega}(\cdot,w)|^2\mathrm{d}\lambda\\
&&\hspace{1cm}\leq_{(\ref{eq:2.5.2})} C  +\sum_{k=k_0}^{\infty} C\cdot 2^{\frac{kq}{\alpha}}\cdot 2^{-kr} \\
&&\hspace{1cm}\leq C+\sum_{k=k_0}^{\infty} C\cdot 2^{(\frac{q}{\alpha}-r)k}.
\end{eqnarray*}
Thus $\int_{\Omega}|K_{\Omega}(\cdot,w)|^2(\log^{+}|K_{\Omega}(\cdot,w)|)^{q}\mathrm{d}\lambda < +\infty$  when $q<\alpha r$. Since $\alpha <\alpha_{\l}(\Omega)$ and $r\in  (0,1)$ can be chosen arbitrarily, we have $$K_{\Omega}(\cdot,w) \in A^2(\log A)^q(\Omega), \, \, \, \forall \, 0<q< \alpha_{l}(\Omega).$$
\end{proof}

\begin{proof}[Proof of Corollary \ref{corollary6.4}]
We infer from Theorem 1.3 that
$$
 \Lambda:=\{ K_{\Omega}(\cdot, w); w\in \Omega  \} \subset    A^2(\log A)^q(\Omega).
$$
Since $A^2(\log A)^q(\Omega)$ is a linear space, we have
$$
\mathrm{Span} \{ \Lambda \} \subset A^2(\log A)^q(\Omega).
$$
If $f\in A^2(\Omega)$ and $f \perp \overline{\mathrm{Span}\{ \Lambda \}}$, then 
$$
f(w)=\int_{\Omega}f(\zeta)\overline{K_{\Omega}(\zeta,w)}\mathrm{d}\lambda(\zeta)=0
$$
for every $w\in \Omega$, i.e., $f \equiv 0$. In other words,  
$\overline{\mathrm{Span}\{ \Lambda \} } =A^2(\Omega)$.  
So $A^2(\log A)^q(\Omega)$ lies dense in $A^2(\Omega)$.
\end{proof}

\section{Proof of Theorem 1.5}

Set $\Omega^{(s)}=\{ (z,w)\in \mathbb{C}^2; |z| < r_2(|w|), |z-c(|w|)|>r_1(|w|)  \}$, where $r_1, r_2$ and $c_s$ are smooth functions on $[1,6]$ such that
\begin{equation*}
r_1(x)= \begin{cases}
3-\sqrt{x-1} , &\text{as $x \to 1+$}  \\
\text{decreasing}& \text{for $x \in [1,2]$}    \\
1 , &\text{for $x \in [2,5]$}  \\
\text{increasing}& \text{for $x \in [5,6]$}    \\
3-\sqrt{6-x} , &\text{as $x \to 6-$} 
\end{cases}
\end{equation*}
\begin{equation*}
r_2(x)= \begin{cases}
3+\sqrt{x-1} , &\text{as $x \to 1+$}  \\
\text{increasing}& \text{for $x \in [1,2]$}    \\
4 , &\text{for $x \in [2,5]$}   \\
\text{decreasing}& \text{for $x \in [5,6]$}   \\
3+\sqrt{6-x} , &\text{as $x \to 6-$} 
\end{cases}
\end{equation*}
and
\begin{equation*}
c_s(x)= \begin{cases}
0 , &\text{for $x \in [1,2]$}  \\
\text{decreasing}& \text{for $x \in [2,3]$}    \\
\frac{1}{2}e^{-2|x-3|^{-s}}-1 , &\text{for $x$ in a small neighborhood of $3$}   \\
\text{increasing}& \text{for $x \in [3,4]$}    \\
1-\frac{1}{2}e^{-2|x-4|^{-s}} , &\text{for $x$ in a small neighborhood of $4$}  \\
\text{decreasing}& \text{for $x \in [4,5]$}   \\
0 , &\text{for $x \in [5,6]$} .
\end{cases}
\end{equation*}
Clearly, $g(x):=x^2(\log^{+} x)^q$ is increasing on $\mathbb{R}_{+}$ and convex when $x>x_0$ for some $x_0 \in \mathbb{R}_{+}$. Define a convex increasing function on $\mathbb{R}_{+}$ as follows
$$
h(x)=\begin{cases}
g(x_0),  & 0 \leq x \leq x_0\\
g(x),  & x>x_0.
\end{cases}
$$

\begin{lemma} 
If $s\in (0,1)$ and $ q\geq \frac{1}{s}-1$, then $\frac{1}{z} \in A^2(\Omega^{(s)})$ while $\frac{1}{z} \notin A^2(\log A)^q(\Omega^{(s)})$. 
\end{lemma} 
\begin{proof} 
Define $\Omega^{(s)}_w:=\{ z \in \mathbb{C}; (z,w)\in \Omega^{(s)} \}$. Clearly, $\Omega^{(s)}_w=\Omega^{(s)}_{|w|}$.

First of all, we have
\begin{eqnarray*}
\int_{\Omega^{(s)}}\frac{1}{|z|^2}\mathrm{d}\lambda(z,w)&=&\int_{\{1\leq|w|\leq 6\}}\mathrm{d}\lambda(w) \int_{\Omega^{(s)}_w}\frac{1}{|z|^2}\mathrm{d}\lambda(z)     \\
&=&2\pi \int_1^6 t \mathrm{d}t \int_{\Omega^{(s)}_t}\frac{1}{|z|^2}\mathrm{d}\lambda(z) \\
&=& 2\pi \left( \int_{ \{ t\in [1,6], |t-3|>\varepsilon, |t-4|>\varepsilon    \}  } +\int_{3-\varepsilon}^{3+\varepsilon} +\int_{4-\varepsilon}^{4+\varepsilon}     \right) t \mathrm{d}t \int_{\Omega^{(s)}_t}\frac{1}{|z|^2}\mathrm{d}\lambda(z) \\
&\leq& C_{\varepsilon}+2\pi \int_{-\varepsilon}^{\varepsilon}(3+t)\mathrm{d}t\int_{\Omega^{(s)}_{3+t}}\frac{1}{|z|^2}\mathrm{d}\lambda+2\pi \int_{-\varepsilon}^{\varepsilon}(4+t)\mathrm{d}t\int_{\Omega^{(s)}_{4+t}}\frac{1}{|z|^2}\mathrm{d}\lambda(z) \\
&\leq& C_{\varepsilon}+4\pi \int_{-\varepsilon}^{\varepsilon}(4+t)\mathrm{d}t \int_{ \Omega^{(s)}_{3+t}\cup \Omega^{(s)}_{4+t}  } \frac{1}{|z|^2}\mathrm{d}\lambda(z)  \\
&\leq& C_{\varepsilon}+4\pi \int_{-\varepsilon}^{\varepsilon}(4+t)\mathrm{d}t \int_{\{ \frac{1}{2}e^{-2|t|^{-s}}<|z|<4  \}} \frac{1}{|z|^2}\mathrm{d}\lambda(z)  \\
&\leq& C_{\varepsilon}+ 8\pi^2 \int_{-\varepsilon}^{\varepsilon} (4+t)(\log 8 +2|t|^{-s})\mathrm{d}t \\
&<& +\infty.
\end{eqnarray*}

On the other hand, since $\Omega^{(s)}_{3+t}\cup \Omega^{(s)}_{4+t} \supset \{ z\in \mathbb{C}; e^{-|t|^{-s}}<|z|<1 \}$, we have
\begin{eqnarray*}
\int_{\Omega^{(s)}}\frac{1}{|z|^2} \left( \log^{+}\frac{1}{|z|}  \right)^q \mathrm{d}\lambda(z,w) &\geq& 2\pi \int_1^6 t\mathrm{d}t \int_{\Omega^{(s)}_t}\frac{1}{|z|^2}\left( \log^{+} \frac{1}{|z|}  \right)^q \mathrm{d}\lambda(z) \\
&\geq & 2\pi \left( \int_{3-\epsilon}^{3+\epsilon}+\int_{4-\epsilon}^{4+\epsilon}   \right)\mathrm{d}t \int_{\Omega^{(s)}_t}\frac{1}{|z|^2}\left( \log^{+} \frac{1}{|z|}  \right)^q \mathrm{d}\lambda(z) \\
&\geq & 2\pi \int_{-\varepsilon}^{\varepsilon}\mathrm{d}t \int_{\Omega^{(s)}_{3+t}\cup \Omega^{(s)}_{4+t}} \frac{1}{|z|^2}\left( \log^{+}\frac{1}{|z|}  \right)^q  \mathrm{d}\lambda(z)\\
&\geq & 2\pi \int_{-\varepsilon}^{\varepsilon}\mathrm{d}t \int_{\{ e^{-|t|^{-s}}<|z|<1\} } \frac{1}{|z|^2}\left( \log\frac{1}{|z|}  \right)^q  \mathrm{d}\lambda(z)\\
&=&4\pi^2 \int_{-\varepsilon}^{\varepsilon}   \frac{1}{q+1} (|t|^{-s})^{q+1}\mathrm{d}t \\
&= & +\infty.
\end{eqnarray*}
\end{proof}

\begin{proof}[Proof of Theorem 1.5]
$(1)$ Suppose on the contrary that $A^2(\log A)^q(\Omega^{(s)}) $ lies dense in $A^2(\Omega^{(s)})$, then there exists a sequence $\{ \widetilde{f}_n \} \subset A^2(\log A)^q(\Omega^{(s)})$ such that $\| \widetilde{f}_n - \frac{1}{z}\|_{L^2(\Omega^{(s)})} \to 0$ $(n \to \infty)$.  Set 
$$
f_n(z,w)=\frac{1}{2\pi}\int_{0}^{2\pi}\widetilde{f}_n(z,e^{i\theta}w) \mathrm{d}\theta.
$$ 

{\it Step 1}. We shall verify that $f_n \in A^2(\log A)^q(\Omega^{(s)})$ and $\|f_n-\frac{1}{z}\|_{L^2(\Omega^{(s)})}\to 0$. 

First of all, we have
\begin{eqnarray*}
\int_{\Omega^{(s)}} |f_n|^2(\log^{+}|f_n|)^q \mathrm{d}\lambda(z,w) &=& \int_{\Omega^{(s)}} g\left( \left|\frac{1}{2\pi}\int_{0}^{2\pi} \widetilde{f_n}(z,e^{i\theta}w)\mathrm{d}\theta    \right| \right)   \mathrm{d}\lambda(z,w)    \\
&\leq& \int_{\Omega^{(s)}}g\left(\frac{1}{2\pi}\int_0^{2\pi}|\widetilde{f}_n(z,e^{i\theta}w)|\mathrm{d}\theta \right) \mathrm{d}\lambda(z,w)\\
&\leq& \int_{\Omega^{(s)}}h\left(\frac{1}{2\pi}\int_0^{2\pi}|\widetilde{f}_n(z,e^{i\theta}w)|\mathrm{d}\theta \right) \mathrm{d}\lambda(z,w)\\
&\leq& \int_{\Omega^{(s)}}\frac{1}{2\pi}\int_{0}^{2\pi}h(|\widetilde{f}_n(z,e^{i\theta}w)|)\mathrm{d}\theta \mathrm{d}\lambda(z,w)\\
&=&\frac{1}{2\pi}\int_{0}^{2\pi} \int_{\Omega^{(s)}} h(|\widetilde{f}_n(z,e^{i\theta}w)|)\mathrm{d}\lambda(z,w)\mathrm{d}\theta \\
&=& \int_{\Omega^{(s)}} h(|\widetilde{f}_n(z,w)|)\mathrm{d}\lambda(z,w) \\
&\leq&g(t_0)|\Omega^{(s)}|+\int_{\{|\widetilde{f}_n|> t_0\}}g(|\widetilde{f}_n(z,w)|) \mathrm{d}\lambda(z,w) \\
&\leq&g(t_0)|\Omega^{(s)}|+\int_{\Omega^{(s)}}|\widetilde{f}_n(z,w)|^2(\log^{+}|\widetilde{f}_n(z,w)|)^q \mathrm{d}\lambda(z,w) \\
&<&\infty,
\end{eqnarray*}
where the fourth inequality follows from Jensen's inequality, and the sixth equality follows from the fact that the inner integral is independent of $\theta$. Thus $f_n \in A^2(\log A)^q(\Omega^{(s)})$. 

Next, we have 
\begin{eqnarray*}
\int_{\Omega^{(s)}}\left|f_n-\frac{1}{z}\right|^2 \mathrm{d}\lambda(z,w)&=&\int_{\Omega^{(s)}} \left|\frac{1}{2\pi}\int_{0}^{2\pi}\left(\widetilde{f}_n(z,e^{i\theta}w)-\frac{1}{z}\right)\mathrm{d}\theta \right|^2 \mathrm{d}\lambda(z,w) \\
&\leq& \frac{1}{2\pi}\int_{\Omega^{(s)}}\int_0^{2\pi}\left| \widetilde{f}_n(z,e^{i\theta}w)-\frac{1}{z}   \right|^2\mathrm{d}\theta \mathrm{d}\lambda(z,w) \\
&=& \frac{1}{2\pi}\int_0^{2\pi}\int_{\Omega^{(s)}}\left| \widetilde{f}_n(z,e^{i\theta}w)-\frac{1}{z}   \right|^2 \mathrm{d}\lambda(z,w)\mathrm{d}\theta \\
&=&\frac{1}{2\pi}\int_0^{2\pi} \left\| \widetilde{f}_n(z,e^{i\theta}w)-\frac{1}{z} \right\|^2_{L^2(\Omega^{(s)})}\mathrm{d}\theta \\
&=& \left\| \widetilde{f}_n(z,w)-\frac{1}{z} \right\|^2_{L^2(\Omega^{(s)})} \to 0 \hspace{0.2cm} (n \to \infty),
\end{eqnarray*}
where the second inequality follows from Schwarz's inequality.
 So $f_n \to \frac{1}{z}$ in $L^2(\Omega^{(s)})$. 

{\it Step 2}. We claim that each $f_n$ is independent of $w$, and  can be extended to a holomorphic function on $\{|z|<4\}$.

In fact, for every $z \in \mathbb{C}$ with $|z|=3$, we conclude that
\begin{eqnarray*}
f_n(z,w)&=&\frac{1}{2\pi}\int_{0}^{2\pi}\widetilde{f}_n(z,e^{i\theta}w) \mathrm{d}\theta\\
&=&\frac{1}{2\pi} \int_{\{|\zeta|=1\}}\frac{\widetilde{f}_n(z,w\zeta)}{\zeta}\mathrm{d}\zeta \\
&=&\frac{1}{2\pi} \int_{\{|\zeta|=|w|\}}\frac{\widetilde{f}_n(z,\zeta)}{\zeta}\mathrm{d}\zeta
\end{eqnarray*}
is independent of $w$ when $1<|w|<6$, in view of Cauchy's theorem. Now for any $w_1, w_2 \in \{  1<|w|<6 \}$, the function $f_n(\cdot,w_1)-f_n(\cdot, w_2)$ is holomorphic, and vanishes on $\{ |z|=3 \}$, we infer from the identity theorem that $f_n(\cdot,w_1)-f_n(\cdot, w_2) \equiv 0$, i.e., $f_n$ is independent of $w$, thus we have $$f_n \in \mathcal{O}(\{ 0<|z|<4 \}).$$ Write $f_n(z) = \sum_{m=-\infty}^{+\infty}a_m^{(n)}z^m$. Since $f_n \in A^2(\log A)^q(\Omega^{(s)}) \subset A^2(\Omega^{(s)})$, we have
\begin{eqnarray*}
+\infty &>& \int_{\Omega^{(s)}} |f_n(z)|^2 \mathrm{d}\lambda(z,w)\\
&=&2\pi \int_1^6 t\mathrm{d}t \int_{\Omega^{(s)}_t} |f_n(z)|^2\mathrm{d}\lambda(z) \\
&\geq &2\pi \int_{-\varepsilon}^{\varepsilon}\mathrm{d}t \int_{\Omega^{(s)}_{3+t}\cup \Omega^{(s)}_{4+t}}|f_n(z)|^2\mathrm{d}\lambda(z) \\
&\geq&  2\pi \int_{-\varepsilon}^{\varepsilon}\mathrm{d}t \int_{\{ e^{-|t|^{-\frac{1}{2}}}<|z|<4\}} \sum_{m=-\infty}^{+\infty} |a_m^{(n)}|^2|z|^{2m} \mathrm{d}\lambda(z) \\
&=& 4\pi^2 \int_{-\varepsilon}^{\varepsilon}\mathrm{d}t \int_{e^{-|t|^{-\frac{1}{2}}}}^4 \sum_{m=-\infty}^{\infty}|a_m^{(n)}|^2\cdot r^{2m+1} \mathrm{d}r \\
&=& 4\pi^2\sum_{m \neq -1} \int_{-\varepsilon}^{\varepsilon}   |a_m^{(n)}|^2\frac{4^{2m+2}-e^{-(2m+2)|t|^{-\frac{1}{2}}}}{2m+2} \mathrm{d}t +4\pi^2 \int_{-\varepsilon}^{\varepsilon}|a^{(n)}_{-1}|^2(\log 4 +|t|^{-\frac{1}{2}})     \mathrm{d}t .
\end{eqnarray*}
When $m \leq -2$, $e^{-(2m+2)|t|^{-\frac{1}{2}}}$ is not integrable, so $a_{m}^{(n)}=0$. Thus we may write $$f_n(z)=\frac{a_{-1}^{(n)}}{z}+g_n(z),$$ where $g_n$ is holomorphic on $\{|z|<4\}$. Notice that $\frac{1}{z}$ is bounded on $\{|z|\geq \varepsilon\}$, so 
$$\frac{a_{-1}^{(n)}}{z} \in A^2(\log A)^q(\Omega^{(s)} \cap \{ |z|\geq \varepsilon \}).$$ 
 Also, since $f_n \in A^2(\log A)^q(\Omega^{(s)})$ and $g_n$ is bounded on $\{ |z|< \varepsilon \}$, it follows from Proposition \ref{prop:3.1}$/(1)$ we have 
$$\frac{a_{-1}^{(n)}}{z} \in A^2(\log A)^q(\Omega^{(s)} \cap \{ |z|<\varepsilon \}).$$
 So $\frac{a_{-1}^{(n)}}{z}$ lies in $A^2(\log A)^q(\Omega^{(s)})$. But we have already known that $\frac{1}{z}$ is not in $A^2(\log A)^q(\Omega^{(s)})$, thus $a_{-1}^{(n)}$ must be zero, and $f_n$ is in fact a holomorphic function on $\{|z|<4\}$, which concludes the proof of the claim. 

Since $f_n \to \frac{1}{z}$ in $L^2(\Omega^{(s)})$, the convergence holds uniformly on 
$\{(z,w_0); |z|=3\}$, where $w_0$ is a fixed point in $\{ 1<|w|<6 \}$. So we have
$$
0=\int_{\{(z,w_0); |z|=3\}} f_n(z) \mathrm{d}z \to \int_{\{(z,w_0); |z|=3\}} \frac{1}{z} \mathrm{d}z=2\pi i,
$$
which is absurd. Thus $A^2(\log A)^q(\Omega^{(s)})$ is not dense in $A^2(\Omega^{(s)})$.

(2) Suppose on the contrary that $K_{\Omega^{(s)}}(\cdot, w) \in A^2(\log A)^q(\Omega^{(s)})$ for every $w\in \Omega^{(s)}$. Then an analogous argument as the proof of Corollary 1.4 shows that $A^2(\log A)^q(\Omega^{(s)})$ is dense in $A^2(\Omega^{(s)})$, a contradiction to $(1)$.
\end{proof}

\section{Appendix}

We shall construct a bounded planar domain $\Omega$ with $\alpha(\Omega)=0$ while $\alpha_l(\Omega)>0$, using Corollary 3.4 of \cite{ChenDensity} and Theorem 1.1 of \cite{CarlesonTotik2004}. Set $D(a,r)=\{|z-a|<r\}$, and $E_r(a)=\overline{D}(a,r)-\Omega$. We denote by $C_l(E)$  the logarithmic capacity of a compact set $E\subset \mathbb{C}$.

\begin{definition}[\cite{CarlesonTotik2004}]
For a compact set $E\subset \mathbb{R}$, $a\in E$ and $\varepsilon>0,$ set 
$$
\mathcal{N}_E(a,\varepsilon)=\{ n\in \mathbb{N}; C_l(E_{2^{-n}}(a)-\overline{D}(a,2^{-n-1}))\geq \varepsilon \cdot 2^{-n}\}, 
$$
and we say that $\mathcal{N}_E(a,\varepsilon)$ is of positive lower density if
$$
\liminf_{N\to \infty} \frac{|\mathcal{N}_E(\varepsilon)\cap \{ 1,2, \cdots, N\}|}{N+1}>0.
$$
\end{definition}

\begin{theorem}[\cite{CarlesonTotik2004}, Theorem 1.1]
For a compact set $E\subset \mathbb{R}$ the Green's function $g_{\overline{\mathbb{C}}-E}(\cdot, \infty)$ is H\"older continuous at 0 if and only if $\mathcal{N}_E(0,\varepsilon)$ is of positive lower density for some $\varepsilon >0$.
\end{theorem}

\begin{definition}[\cite{ChenDensity}]
For $\varepsilon>0, 0<\lambda<1$ and $ \gamma >1$ we set
\begin{eqnarray*}
&&\mathcal{N}_a(\varepsilon, \lambda,\gamma):= \{ n \in \mathbb{Z}^{+} ; C_l(E_{\lambda^n}(a))\geq \varepsilon \lambda^{\gamma n}   \} \\
&&\mathcal{N}_a^n(\varepsilon, \lambda,\gamma):=\mathcal{N}_a(\varepsilon, \lambda,\gamma) \cap \{1,2,\dots,n \}.
\end{eqnarray*}
We define the $\gamma$-capacity density of $\partial \Omega$ at $a$ by
$$\mathcal{D}_a(\varepsilon,\lambda,\gamma)=\liminf_{n\to \infty}\frac{\sum_{k\in \mathcal{N}_a^n(\varepsilon,\lambda,\gamma)}k^{-1}}{\log n},$$
and the  $\gamma$-capacity density of $\partial \Omega$ by $$\mathcal{D}_W(\varepsilon, \lambda, \gamma)=\liminf_{n\to \infty}\frac{\inf_{a\in \partial \Omega}  \sum_{k\in \mathcal{N}_a^n(\varepsilon,\lambda,\gamma)}k^{-1}   }{\log n}.$$
\end{definition}

\begin{theorem}[\cite{ChenDensity}, Corollary 3.4]
If $\mathcal{D}_W(\epsilon, \lambda, \gamma)>0$ for some $\epsilon, \lambda, \gamma$, then there exists $\beta>0$ such that 
$$
\phi_K(z)\leq (-\log \delta(z))^{-\beta}
$$ 
for all $z$ sufficiently close to $\partial \Omega$, where $\phi_K$ denotes the capacity potential of a compact subset $K$ relative to $\Omega$.
\end{theorem}

\begin{example}
Let $\Omega = D(0,3)-E,$ where $E=  \bigcup_{n=0}^{\infty}[2^{-n},2^{-n}+2^{-2 n}]\cup \{0\} $. Then $\alpha(\Omega)=0$ and $\alpha_l(\Omega)>0$.
\end{example}
\begin{proof}
We first calculate $\mathcal{N}_a^n(\frac{1}{16}, \frac{1}{2},2)$:

The case $a=0$ is simple, for
$$C_l(E_{2^{-n}}(0))\geq C_l([2^{-n-1},2^{-n-1}+2^{-2(n+1)}])= \frac{1}{4}\cdot 2^{-2(n+1)}= \frac{1}{16}\cdot 2^{-2n}.$$

The case $a \in [2^{-n_0},2^{-n_0}+2^{-2 n_0}]$ is divided into three parts:

(i) If $n\leq n_0 -1$, then $0\in \overline{D}(a,2^{-n})$. It is easy to see that $E_{2^{-n}}(0) \subset E_{2^{-n}}(a)$. Thus $$C_l(E_{2^{-n}}(a)) \geq C_l(E_{2^{-n}}(0)) \geq \frac{1}{16}\cdot 2^{-2n};$$

(ii) If $n_0 \leq n \leq 2n_0$, then $[2^{-n_0}, 2^{-n_0}+2^{-2n_0}] \subset \overline{D}(a,2^{-n})$, so $$C_l(E_{2^{-n}}(a))\geq \frac{1}{4}\cdot2^{-2n_0}\geq \frac{1}{4}\cdot 2^{-2 n}\geq \frac{1}{16}\cdot 2^{-2n} ; $$

(iii) If $n \geq 2n_0+1$, then $\overline{D}(a,2^{-n}) \cap [2^{-n_0},2^{-n_0}+2^{-2n_0}]$ contains an interval with length $2^{-n}$, so $$C_l(E_{2^{-n}}(a))\geq \frac{1}{4}\cdot 2^{-n}\geq \frac{1}{16}\cdot 2^{-2n}.$$

If $a\in \partial D(0,3)$, then $E_{2^{-n}}(a)$ contains an interval with length $2^{-n+1}$. Thus $$C_l(E_{2^{-n}}(a))\geq \frac{1}{4} \cdot 2^{-n+1}\geq \frac{1}{16}\cdot 2^{-2n}.$$

Hence $\mathcal{N}_a^{n}(\frac{1}{16},\frac{1}{2},2)=\{1,2,\dots , n \}$ for all $a\in \partial \Omega$, which implies $\mathcal{D}_W(\frac{1}{16},\frac{1}{2},2)>0$. By Theorem 5.2, we have $\alpha_l(\Omega)>0$.

On the other hand, we have $$C_l(E_{2^{-n}}(0)-\overline{D}(0,2^{-n-1}))=\frac{1}{4}\cdot 2^{-2(n+1)}.$$ Hence for any $\varepsilon >0$, there exists an integer $n_0$ such that $$C_l(E_{2^{-n}}(0)-\overline{D}(0,2^{-n-1})) < \varepsilon \cdot 2^{-n}$$ for every $n > n_0$. From Theorem 5.1, we know that $g_{\overline{\mathbb{C}}-E}(z,\infty)$ is not H\"older continuous at $0$. Suppose  on the contrary that there exists a function $\phi \in SH^{-}(\Omega) \cap C(\Omega)$ such that $-\phi \lesssim \delta^{\beta}$ for some $\beta>0$. It is easy to see from the maximum principle that
$$
g_{\overline{\mathbb{C}}-E}(z, \infty) \lesssim -\phi(z) \lesssim \delta^{\beta}(z) \leq |z-0|^{\beta}
$$
in a neighborhood of $0$, which is a contradiction.
\end{proof}

\end{document}